\documentclass[a4paper,11pt]{article}
\usepackage[]{graphicx}
\usepackage[]{amsmath,amssymb}
\usepackage[]{amsthm,enumerate}
\usepackage{verbatim,bm}
\usepackage{color}
\usepackage{a4wide}

\newtheorem{theorem}{Theorem}[section]
\newtheorem{lemma}[theorem]{Lemma}
\newtheorem{corollary}[theorem]{Corollary}
\theoremstyle{definition}
\newtheorem{proposition}[theorem]{Proposition}

\newtheorem{example}[theorem]{Example}
\newcommand{\defn}[1]{{\em #1}}
\theoremstyle{remark}

\title{Commutative association schemes obtained from twin prime powers, Fermat primes,  Mersenne primes}
\date{\today}

\author{
 Hadi Kharaghani\thanks{Department of Mathematics and Computer Science, University of Lethbridge,
Lethbridge, Alberta, T1K 3M4, Canada.  \texttt{kharaghani@uleth.ca}} 
\and
 Sho Suda\thanks{Department of Mathematics, National Defense Academy of Japan, Yokosuka,  239-8686, Japan. \texttt{ssuda@nda.ac.jp}}
}

\begin{document}
\maketitle
\abstract{
For prime powers $q$ and $q+\varepsilon$ where $\varepsilon\in\{1,2\}$,  
an affine resolvable design 
from $\mathbb{F}_q$ and Latin squares from $\mathbb{F}_{q+\varepsilon}$ yield a set of symmetric designs if $\varepsilon=2$ and a set of symmetric group divisible designs if $\varepsilon=1$.
We show that these designs derive commutative association schemes, and determine their eigenmatrices.  
}


\section{Introduction}
By \defn{twin prime powers}, we mean a pair of integers $q$ and $q+2$, each of which is a prime power. 
Twin prime powers have been used to construct Hadamard difference sets and Hadamard matrices. 
Let $q$ and $q+2$ be twin prime powers, and let $\mathbb{F}_q$ and $\mathbb{F}_{q+2}$ be the finite fields of order $q$ and $q+2$ respectively. 
Let $\chi_q^{(2)}$ and $\chi_{q+2}^{(2)}$ be the quadratic characters of $\mathbb{F}_q$ and $\mathbb{F}_{q+2}$ respectively. 
Define $X=\mathbb{F}_{q}\times\mathbb{F}_{q+2}$ and  
\begin{align*}
R_0&=\{(x,x)\mid x\in X \},\\
R_1&=\{((x_1,y_1),(x_2,y_2))\in X\times X \mid \chi_q^{(2)}(x_1-x_2)\chi_{q+2}^{(2)}(y_1-y_2)=1 \},\\
R_2&=\{((x_1,y_1),(x_2,y_2))\in X\times X \mid \chi_q^{(2)}(x_1-x_2)\chi_{q+2}^{(2)}(y_1-y_2)=-1 \},\\
R_3&=\{((x_1,y),(x_2,y))\in X\times X \mid x_1\neq x_2 \},\\
R_4&=\{((x,y_1),(x,y_2))\in X\times X \mid y_1\neq y_2 \}. 
\end{align*} 
Then it is easy to see that $(X,\{R_i\}_{i=0}^4)$ is a translation commutative association scheme with $4$ classes. 
Note that by denoting $R_i(x)=\{y\in X\mid (x,y)\in R_i\}$ and ${\bf0}=(0,0)\in X$, $R_0({\bf0})\cup R_1({\bf0})\cup R_3({\bf0})$ is a twin prime powers difference set with parameters $(q(q+2),(q^2+2q-1)/2,(q+3)(q-1)/4)$ \cite{SS}, see also \cite{B71}. 

In this paper, we make use of twin prime powers to construct a translation commutative association scheme with vertex set $\mathbb{F}_{q+2}\times \mathbb{F}_{q}\times \mathbb{F}_{q}$ and with $q+3$ classes. 
Our association schemes are based on the construction of symmetric designs due to Wallis \cite{W} obtained from an affine resolvable design and a Latin square. 
See also \cite[Theorem~5.23]{S} for the construction. 
Applying this result to twin prime powers, we obtain a collection of symmetric $(q^2(q+2),q(q+1),q)$-designs with incidence matrices $N_{\beta}$ for $\beta\in\mathbb{F}_{q+2}^*$, where $\mathbb{F}_{r}^*=\mathbb{F}_{r}\setminus\{0\}$ for any prime power $r$. 
We will show that $N_{\beta}$, $\beta\in\mathbb{F}_{q+2}^*$, are commuting. 
Note that these symmetric designs are an example of \defn{linked systems of symmetric group divisible designs of type II} in \cite{KS2} and that a similar result based on these symmetric designs for the finite field $\mathbb{F}_{2^m}$ is obtained in \cite{KSS}.  
The matrices $N_{\beta}$ ($\beta\in\mathbb{F}_{q+2}^*$) share $(J_{q+2}-I_{q+2})\otimes I_{q^2}$ and it holds that $\sum_{\beta\in\mathbb{F}_{q+2}^*}N_\beta-(q+1)(J_{q+2}-I_{q+2})\otimes I_{q^2}=I_{q+2}\otimes J_{q^2}$, 
where $I_n,J_n$ denote the identity matrix of order $n$, and the all-ones matrix of order $n$,  respectively.    
One of our main results is Theorem~\ref{thm:as} that the matrices $I_{(q+2)q^2}, (J_{q+2}-I_{q+2})\otimes I_{q^2},I_{q+2}\otimes (J_{q^2}-I_{q^2})$, $N_\beta-(J_{q+2}-I_{q+2})\otimes I_{q^2}$ ($\beta\in\mathbb{F}_{q+2}^*$) form a commutative association scheme with $q+3$ classes. 

Furthermore, extending the idea of the construction of symmetric designs by Wallis, we obtain a set of symmetric group divisible designs from prime powers $q$ and $q+1$. 
Note that by a solution of famous Catalan's conjecture \cite{P}, $q$ and $q+1$ are both prime powers if and only if one of the following holds:
\begin{enumerate}
\item $q=2^{2^m}$ for some positive integer $m$ and $q+1$ is a prime number, which is called a \defn{Fermat prime}. 
\item $q+1=2^m$ for some positive integer $m$ and $q$ is a prime, which is called a \defn{Mersenne prime}. 
\item $q=8$. 
\end{enumerate}
In a similar manner to twin prime powers case, we obtain a commutative association scheme with $q+4$ classes in Theorem~\ref{thm:mfas}.

The organization of the paper is as follows. 
In Section~\ref{sec:pre}, we recall the definition of symmetric designs, symmetric group divisible designs, commutative association schemes and their eigenmatrices. We also prepare the results on auxiliary matrices and Latin squares obtained from finite fields needed later. 
We then construct commutative association schemes with $q+3$ classes from twin prime powers in Section~\ref{sec:astp}, and those with $q+4$ classes from Fermat primes or Mersenne primes in Section~\ref{sec:asmfp}. 
\section{Preliminaries}\label{sec:pre}
\subsection{Symmetric designs, symmetric group divisible designs}
Let $m,n\geq 2$ be integers. 
A (\emph{square}) \emph{group divisible design with parameters $(v,k,m,n,\lambda_1,\lambda_2)$} is a pair $(V,\mathcal{B})$, where $V$ is a finite set of  $v$ elements called points, and $\mathcal{B}$ a collection of $k$-element subsets of $V$ called blocks with $|\mathcal{B}|=v$, in which the point set $V$ is partitioned into $m$ classes of size $n$, such that two distinct points from one class occur
together in $\lambda_1$ blocks, and two points from different classes occur together in exactly $\lambda_2$ blocks. 
A group divisible design is said to be \emph{symmetric} (or to have the \emph{dual property}) if its dual, that is 
the structure gotten by interchanging the roles of points and blocks, is again a group divisible design with the same parameters. Refer to \cite{B} for the details. 
A group divisible design is said to be \defn{proper} if $\lambda_1\neq \lambda_2$ 
and \defn{improper} if $\lambda_1=\lambda_2$. 
In the improper case,  
we set $\lambda=\lambda_1=\lambda_2$. 
Improper symmetric group divisible designs are known as \defn{symmetric $2$-$(v,k,\lambda)$ designs} or \defn{symmetric designs with parameters $(v,k,\lambda)$}. 

The \defn{incidence matrix} of a symmetric group divisible design $(V,\mathcal{B})$ is a $v\times v$ $(0,1)$-matrix $A$ with rows and columns indexed by $\mathcal{B},V$ respectively such that for $x\in V, b\in \mathcal{B}$, 
\begin{align*}
A_{b,x}=\begin{cases}
1 & \text{ if } x\in b,\\
0 & \text{ if } x\not\in b.
\end{cases}
\end{align*}
Let $A$ be the incidence matrix of a symmetric group divisible design with parameters $(v,k,m,n,\lambda_1,\lambda_2)$. 
Then, after reordering the elements of $V$ and $\mathcal{B}$ appropriately,  
\begin{align}\label{eq:gdd}
A A^\top=A^\top A=k I_v+\lambda_1(I_m\otimes J_n-I_v)+\lambda_2(J_v-I_m\otimes J_n),  
\end{align}
where $A^\top$ is the transpose of $A$. 
We also refer to a $v\times v$ $(0,1)$-matrix $A$ satisfying \eqref{eq:gdd} as a symmetric group divisible design. 
A $v\times v$ $(0,1)$-matrix $A$ satisfying $A A^\top=A^\top A=k I_v+\lambda(J_v-I_v)$ is also referred to as a symmetric design.

\subsection{Association schemes}
Let $n$ be a positive integer. 
Let $X$ be a finite set and $R_i$ ($i\in\{0,1,\ldots,n\}$) be a nonempty subset of $X\times X$. 
The \emph{adjacency matrix} $A_i$ of the graph with vertex set $X$ and edge set $R_i$ is a $(0,1)$-matrix indexed by $X$ such that $(A_i)_{xy}=1$ if $(x,y)\in R_i$ and $(A_i)_{xy}=0$ otherwise. 
A \emph{commutative association scheme} with $n$ classes is a pair $(X,\{R_i\}_{i=0}^n)$ satisfying the following:
\begin{enumerate}[(\text{AS}1)]
\item $A_0=I_{|X|}$.
\item $\sum_{i=0}^n A_i = J_{|X|}$.
\item $A_i^\top\in\{A_1,\ldots,A_n\}$ for any $i\in\{1,\ldots,n\}$.
\item For all $i$, $j$, $A_i A_j$ is a linear combination of $A_0,A_1,\ldots,A_n$.
\item $A_i A_j=A_j A_i$ for all $i,j$. 
\end{enumerate}
We will also refer to $(0,1)$-matrices $A_0,A_1,\ldots,A_n$ satisfying (AS1)-(AS5) as a commutative association scheme. 
If a commutative association scheme satisfies that $A_i^\top=A_i$ for any $i\in\{1,\ldots,n\}$, then the association scheme is said to be \defn{symmetric}.   
The vector space spanned by $A_i$'s forms a commutative algebra, denoted by $\mathcal{A}$ and called the \emph{Bose-Mesner algebra}.
There exists a basis of $\mathcal{A}$ consisting of primitive idempotents, say $E_0=(1/|X|)J_{|X|},E_1,\ldots,E_n$. 
Note that $E_i$ is the projection onto a maximal common eigenspace of $A_0,A_1,\ldots,A_n$. 
Since  $\{A_0,A_1,\ldots,A_n\}$ and $\{E_0,E_1,\ldots,E_n\}$ are two bases in $\mathcal{A}$, there exist the change-of-bases matrices $P=(p_{ij})_{i,j=0}^n$, $Q=(q_{ij})_{i,j=0}^n$ so that
\begin{align*}
A_j=\sum_{i=0}^n p_{ij}E_i,\quad E_j=\frac{1}{|X|}\sum_{i=0}^n q_{ij}A_i.
\end{align*}
The matrices $P$ or $Q$ are said to be {\it first or second eigenmatrices} respectively. 
An association scheme is said to be {\it self-dual} if $P=\bar{Q}$ for suitable rearrangement the indices of the adjacency matrices and the primitive idempotents, where $\bar{Q}=(\bar{q_{ij}})_{i,j=0}^n$.   

The association scheme is a {\it translation association scheme} if the vertex set $X$ has the structure of an additively written abelian group, and for all $x,y,z\in X$ and  
for all $i\in\{0,1,\ldots,n\}$, 
\begin{align*}
(x,y)\in R_i\Rightarrow (x+z,y+z)\in R_i.
\end{align*}

For translation association schemes, the first eigenmatrix is calculated by the characters as follows. 
For $i\in \{0,1,\ldots,n\}$ set $N_i=R_i(0)=\{x\in X\mid(0,x)\in R_i\}$. 
For each character $\chi$ of $X$ we have 
\begin{align*}
A_i\chi=\left(\sum_{x\in N_i}\chi(x)\right)\chi.
\end{align*}
Letting $X^*$ be the dual group of $X$, set $N_j^*=\{\eta\in X^*\mid E_j\eta=\eta\}$ for $j\in\{0,1,\ldots,n\}$. 
Then the first eigenmatrix of the translation association scheme is expressed as 
\begin{align*}
p_{ij}=\sum_{x\in N_i}\chi(x) \text{ for } \chi\in N_j^*.
\end{align*}

\subsection{Auxiliary matrices from finite fields}
We denote by $\mathbb{F}_q$ the finite field of order $q$. 
Let $H_q$ be the multiplicative table of $\mathbb{F}_q$, i.e., for $\alpha,\beta\in\mathbb{F}_q$, the $(\alpha,\beta)$-entry of $H_q$ is $\alpha \beta$. 
Then the matrix $H_q$ is a \defn{generalized Hadamard matrix $GH(q,1)$ over the additive group of $\mathbb{F}_q$}.  
A \emph{generalized Hadamard matrix $GH(q,1)$ over the additive group of $\mathbb{F}_q$}  is a $q\times q$ matrix $H=(h_{i,j})_{i,j=1}^{q}$ with entries from $\mathbb{F}_q$ 
such that for all distinct $i,k\in\{1,2,\ldots,q\}$, the multiset $\{h_{ij}-h_{kj}\mid 1\leq j\leq q\}$ contains exactly one time of each element of $\mathbb{F}_q$. 

From $H_q$, we have $q$ auxiliary matrices; for each $\alpha\in \mathbb{F}_q$, define a $q^2\times q^2$ $(0,1)$-matrix $C_\alpha$ whose rows and columns indexed by $\mathbb{F}_q\times \mathbb{F}_q$ to be a $q\times q$ block matrix with rows and columns indexed by $\mathbb{F}_{q}$ whose $(\alpha',\alpha'')$-block is $\phi(\alpha(-\alpha'+\alpha''))$;
\begin{align*}
C_\alpha=(\phi(\alpha(-\alpha'+\alpha'')))_{\alpha',\alpha''\in\mathbb{F}_q}, 
\end{align*}
where $\phi$ is a permutation representation of the additive group of $\mathbb{F}_q$ defined as follows.  
Letting $q=p^m$ for a prime $p$, we regard the additive group $\mathbb{F}_q$ as $\mathbb{F}_p^m$. 
Let $r_p$ be a $p\times p$ circulant matrix with the first row $(0,1,0,\ldots,0)$, and define a group homomorphism $\phi$ from the additive group $\mathbb{F}_{q}$ to $GL_{q}(\mathbb{R})$ as $\phi((x_i)_{i=1}^m)= \otimes_{i=1}^m (r_{p})^{x_i}$, where $\otimes$ is the Kronecker product.  
Furthermore, letting $x,y$ be indeterminates, we set $C_x=O_{q^2}$ and $C_{y}=I_q\otimes J_q$, where $O_{q^2}$ is the zero matrix of order $q^2$. 
We say that the matrices $C_a$ ($a\in\mathbb{F}_q\cup\{y\}$) are auxiliary matrices. 
From the auxiliary matrices $C_a$ ($a\in\mathbb{F}_q\cup\{y\}$), one can obtain an affine resolvable design, see \cite[Section~5]{S} for more details. 
See also \cite{BJL}. 
We use the following properties of $C_{a}$ ($a\in\mathbb{F}_q\cup\{y\}$) in subsequent sections.  
\begin{lemma}\label{lem:mfc}\label{lem:m1fc}
\begin{enumerate}
\item $\sum_{a\in\mathbb{F}_q}C_a=q I_{q^2}+(J_{q}-I_q)\otimes J_q$ and $\sum_{a\in\mathbb{F}_q\cup\{y\}}C_a=q I_{q^2}+J_{q^2}$.
\item For any $a\in\mathbb{F}_q\cup\{y\}$, $(C_a)^2=q C_a$.
\item For any distinct $a,a'\in\mathbb{F}_q\cup\{y\}$, $C_aC_{a'}=J_{q^2}$.
\item For $a\in\mathbb{F}_q\cup\{y\}$, $J_{q^2}C_{a}=C_{a}J_{q^2}=qJ_{q^2}$. 
\item For $a\in\mathbb{F}_q\cup\{y\}$, $(I_q\otimes J_q)C_{a}=C_{a}(I_q\otimes J_q)=J_{q^2}$. 
\end{enumerate}
\end{lemma}
\begin{proof}
See \cite[Lemma~2.8]{KS} for the proofs of (i), (ii), (iii). 
(iv) and (v)
 follow from the fact that each $C_a$ is a $q\times q$ block matrix whose blocks are permutation matrices of order $q$.  
\end{proof}

\subsection{Latin squares from finite fields}
Let $\mathbb{F}_{q}$ be the finite field of order $q$. 
Let $H_{q}$ be the multiplicative table of $\mathbb{F}_{q}$.  
From $H_{q}$, we have $q-1$ Latin squares on $\mathbb{F}_q$; for each $\beta\in \mathbb{F}_{q}$, define $L_\beta$ as 
\begin{align*}
L_\beta=(\beta(-\beta'+\beta''))_{\beta',\beta''\in\mathbb{F}_{q}}.
\end{align*}
Note that $L_\beta$ ($\beta\in \mathbb{F}_{q}^*$) form a complete set of mutually suitable Latin squares (MSLS) on $\mathbb{F}_q$; Latin squares $L_1,L_2$ of the same order  are said to be \emph{suitable} if every superimposition of each row of $L_1$ on each row of $L_2$ results in
only one element of the form $(a,a)$, and a set of Latin squares in which every distinct pair of Latin squares
is suitable is called \emph{mutually suitable Latin squares}. 
Note that the existence of mutually suitable Latin squares are equivalent to the existence of mutually orthogonal Latin squares of the same order \cite[Lemma 9]{HKO} and that mutually suitable Latin squares are also called mutually UFS Latin squares in \cite{KS}.

For $\beta\in\mathbb{F}_{q}^*$, define disjoint permutation matrices $P_{\beta,\gamma}$ ($\gamma\in \mathbb{F}_{q}$) by $L_{\beta}=\sum_{\gamma\in \mathbb{F}_q}\gamma P_{\beta,\gamma}$. 
We prepare the following lemma for the permutation matrices $P_{\beta,\gamma}$ ($\beta\in\mathbb{F}_{q}^*,\gamma\in\mathbb{F}_q$). 
We denote the $(a,b)$-entry of a matrix $X$ by $X(a,b)$. 
\begin{lemma}\label{lem:msl}
\begin{enumerate}
\item For $\beta,\beta'\in\mathbb{F}_{q}^*$ and $\gamma\in\mathbb{F}_{q}$, $P_{\beta\beta',\gamma\beta'}=P_{\beta,\gamma}$. 
\item For $\beta,\beta'\in\mathbb{F}_{q}^*$ and $\gamma,\gamma'\in\mathbb{F}_{q}$, $P_{\beta,\gamma}P_{\beta',\gamma'}=P_{\beta \beta',\beta\gamma'+ \beta'\gamma}$. 
\item For $\beta,\beta'\in\mathbb{F}_{q}^*$, 
\begin{align*}
\sum_{\gamma\in\mathbb{F}_{q}}P_{\beta\beta',(\beta+\beta')\gamma}=\begin{cases}qI_{q}&\text{ if } \beta+\beta'=0,\\J_{q}&\text{ if } \beta+\beta'\neq 0.\end{cases}
\end{align*}
\item For $\beta,\beta'\in\mathbb{F}_{q}^*$, 
\begin{align*}
\sum_{\gamma,\gamma'\in\mathbb{F}_{q}^*,\gamma\neq \gamma'}P_{\beta\beta',\beta\gamma'+\beta'\gamma}=\begin{cases}(q-2)(J_{q}-I_{q})&\text{ if } \beta+\beta'=0,\\2I_{q}+(q-3)J_{q}&\text{ if } \beta+\beta'\neq 0.\end{cases}
\end{align*}
\end{enumerate}
\end{lemma}
\begin{proof}
Let $\beta,\beta'\in\mathbb{F}_{q}^*$ and $\gamma,\gamma'\in\mathbb{F}_{q}$. 
By the definition of $P_{\beta,\gamma}$, the $(a,b)$-entry of $P_{\beta,\gamma}$ equals to $1$ if and only if $\beta(-a+b)=\gamma$ for $a,b\in\mathbb{F}_{q}$. 
Thus (i) follows. 
For (ii), $P_{\beta,\gamma}P_{\beta',\gamma'}$ is a permutation matrix and 
\begin{align*}
(P_{\beta,\gamma}P_{\beta',\gamma'})(a,b)=1 &\Leftrightarrow \exists c\in\mathbb{F}_{q} \text{ such that } (P_{\beta,\gamma})(a,c)=(P_{\beta',\gamma'})(c,b)=1\\
&\Leftrightarrow \exists c\in\mathbb{F}_{q} \text{ such that } \beta(-a+c)=\gamma, \beta'(-c+b)=\gamma'\\
&\Leftrightarrow \exists c\in\mathbb{F}_{q} \text{ such that } c=a+\frac{\gamma}{\beta}=b-\frac{\gamma'}{\beta'}\\
&\Leftrightarrow \beta\beta'(-a+b)=\beta\gamma'+\beta'\gamma \\
&\Leftrightarrow P_{\beta \beta',\beta\gamma'+ \beta'\gamma}(a,b)=1. 
\end{align*}
Therefore (ii) holds. 

For (iii), if $\beta+\beta'=0$, then by $P_{\beta'',0}=I_{q}$ for any $\beta''\in \mathbb{F}_{q}^*$, 
\begin{align*}
\sum_{\gamma\in\mathbb{F}_{q}}P_{\beta\beta',(\beta+\beta')\gamma}&=\sum_{\gamma\in\mathbb{F}_{q}}P_{\beta\beta',0}=\sum_{\gamma\in\mathbb{F}_{q}}I_{q}=qI_{q}.
\end{align*}
If $\beta+\beta'\neq 0$, then by (i) and by the fact that $\sum_{\gamma\in\mathbb{F}_{q}}P_{\beta'',\gamma}=J_{q}$ for any $\beta''\in\mathbb{F}_{q}^*$, 
\begin{align*}
\sum_{\gamma\in\mathbb{F}_{q}}P_{\beta\beta',(\beta+\beta')\gamma}&=\sum_{\gamma\in\mathbb{F}_{q}}P_{\frac{\beta\beta'}{\beta+\beta'},\gamma}=J_{q}.
\end{align*}

For (iv), 
\begin{align*}
&\sum_{\gamma,\gamma'\in\mathbb{F}_{q}^*,\gamma\neq \gamma'}P_{\beta\beta',\beta\gamma'+\beta'\gamma}\\
&=\left(\sum_{\gamma\in\mathbb{F}_{q}}P_{\beta,\gamma}\right)\left(\sum_{\gamma'\in\mathbb{F}_{q}}P_{\beta',\gamma'}\right)-\sum_{\gamma\in\mathbb{F}_{q}}P_{\beta\beta',(\beta+\beta')\gamma}
 -\sum_{\gamma\in\mathbb{F}_{q}^*}P_{\beta\beta',\beta'\gamma}-\sum_{\gamma'\in\mathbb{F}_{q}^*}P_{\beta\beta',\beta\gamma'}\\
&=\left(\sum_{\gamma\in\mathbb{F}_{q}}P_{\beta,\gamma}\right)\left(\sum_{\gamma'\in\mathbb{F}_{q}}P_{\beta',\gamma'}\right)-\sum_{\gamma\in\mathbb{F}_{q}}P_{\beta\beta',(\beta+\beta')\gamma}
 -\sum_{\gamma\in\mathbb{F}_{q}^*}P_{\beta,\gamma}-\sum_{\gamma'\in\mathbb{F}_{q}^*}P_{\beta',\gamma'} \text{ (by (i))}\\
&=qJ_{q}-\sum_{\gamma\in\mathbb{F}_{q}}P_{\beta\beta',(\beta+\beta')\gamma}
 -2(J_{q}-I_{q}) \text{ (by $\sum_{\gamma\in\mathbb{F}_{q}}P_{\beta,\gamma}=J_{q},P_{\beta,0}=I_{q}$)}\\
&=\begin{cases}(q-2)(J_{q}-I_{q})&\text{ if } \beta+\beta'=0,\\2I_{q}+(q-3)J_{q}&\text{ if } \beta+\beta'\neq 0.\end{cases} \text{ (by (ii))}\qedhere
\end{align*}
\end{proof}

\section{Association schemes obtained from twin prime powers}\label{sec:astp}
In this section we use twin prime powers $q$ and $q+2$ to construct a set of symmetric designs and derive a commutative association scheme from the symmetric designs. 
\subsection{Symmetric designs}
Let $q,q+2$ be twin prime powers. 
Fix a bijection $\varphi:\mathbb{F}_{q+2}\rightarrow \mathbb{F}_q\cup\{x,y\}$ such that $\varphi(0)=x$. 
Consider a Latin square obtained from $L_\beta$ by replacing entries with the image of $\varphi$, which we denote by $L_{\varphi(\beta)}$. 
Recall that we denote the $(u,v)$-entry of an array $L$ by $L(u,v)$. 
Then for $\beta,\beta',\beta''\in\mathbb{F}_{q+2}$, $L_{\varphi(\beta)}(\beta',\beta'')=\varphi(L_{\beta}(\beta',\beta''))$. 

We now construct symmetric designs from auxiliary matrices for $\mathbb{F}_q$ and mutually suitable Latin squares for $\mathbb{F}_{q+2}$. 
For $\beta\in\mathbb{F}_{q+2}^*$, we define a $(q+2)q^2\times (q+2)q^2$ $(0,1)$-matrix $N_\beta$ to be a $(q+2)\times (q+2)$ block matrix with rows and columns indexed by $\mathbb{F}_{q+2}$ whose $(\beta',\beta'')$-block matrix is $C_{L_{\varphi(\beta)}(\beta',\beta'')}$;  
\begin{align*}
N_\beta=(C_{L_{\varphi(\beta)}(\beta',\beta'')})_{\beta',\beta''\in\mathbb{F}_{q+2}}=\sum_{\gamma\in\mathbb{F}_{q+2}}P_{\beta,\gamma}\otimes C_{\varphi(\gamma)}. 
\end{align*}

\begin{proposition}\label{prop:1}
\begin{enumerate}
\item For any $\beta\in\mathbb{F}_{q+2}^*$, $L_\beta^\top=L_{-\beta}$.   
\item For any $\beta\in\mathbb{F}_{q+2}^*$, $N_{\beta}N_{-\beta}=q^2 I_{(q+2)q^2}+q J_{(q+2)q^2}$. 
\item For any $\beta,\beta'\in\mathbb{F}_{q+2}^*$ such that $\beta+\beta'\neq 0$, $N_{\beta} N_{\beta'}= q N_{\frac{\beta\beta'}{\beta+\beta'}}+2 I_{q+2}\otimes J_{q^2}+(q-1)J_{(q+2)q^2}$. 
\item For any $\beta\in\mathbb{F}_{q+2}^*$, $N_\beta(I_{q+2}\otimes J_{q^2})=(I_{q+2}\otimes J_{q^2})N_\beta=q(J_{(q+2)q^2}-I_{q+2}\otimes J_{q^2})$. 
\item $\sum_{\beta\in\mathbb{F}_{q+2}^*}N_\beta=(J_{q+2}-I_{q+2})\otimes (q I_{q^2}+J_{q^2})$.
\end{enumerate}
\end{proposition}
\begin{proof}
(i) is easy to see, and (iv)  follow from Lemma~\ref{lem:mfc}(iv). 
(ii) is done \cite{W}. 
We prove (ii) as well as (iii) in a same manner.  
For $\beta,\beta'\in\mathbb{F}_{q+2}^*$, by Lemma~\ref{lem:mfc} (ii), (iii), 
\begin{align}
N_\beta N_{\beta'}&=\sum_{\gamma,\gamma'\in\mathbb{F}_{q+2}}P_{\beta\beta',\beta\gamma'+\beta'\gamma}\otimes C_{\varphi(\gamma)}C_{\varphi(\gamma')}\nonumber\\
&=\sum_{\gamma\in\mathbb{F}_{q+2}}P_{\beta\beta',(\beta+\beta')\gamma}\otimes (C_{\varphi(\gamma)})^2+\sum_{\gamma,\gamma'\in\mathbb{F}_{q+2},\gamma\neq \gamma'}P_{\beta\beta',\beta\gamma'+\beta'\gamma}\otimes C_{\varphi(\gamma)}C_{\varphi(\gamma')}\nonumber\\
&=q\sum_{\gamma\in\mathbb{F}_{q+2}}P_{\beta\beta',(\beta+\beta')\gamma}\otimes C_{\varphi(\gamma)}+\sum_{\gamma,\gamma'\in\mathbb{F}_{q+2}^*,\gamma\neq \gamma'}P_{\beta\beta',\beta\gamma'+\beta'\gamma}\otimes J_{q^2}. \label{eq:nn} 
\end{align}
\eqref{eq:nn} is calculated depending on whether $\beta+\beta'=0$ or not as follows.
If $\beta+\beta'=0$, then 
\begin{align*}
\eqref{eq:nn}&=q\sum_{\gamma\in\mathbb{F}_{q+2}}P_{\beta\beta',0}\otimes C_{\varphi(\gamma)}+\sum_{\gamma,\gamma'\in\mathbb{F}_{q+2}^*,\gamma\neq \gamma'}P_{\beta\beta',\beta\gamma'+\beta'\gamma}\otimes J_{q^2}\\ 
&=qI_{q+2}\otimes \sum_{\gamma\in\mathbb{F}_{q+2}}C_{\varphi(\gamma)}+q(J_{q+2}-I_{q+2})\otimes J_{q^2}\text{ (by Lemma~\ref{lem:msl} (iv))}\\
&=qI_{q+2}\otimes (qI_{q^2}+J_{q^2})+q(J_{q+2}-I_{q+2})\otimes J_{q^2}\text{ (by Lemma~\ref{lem:mfc} (i))}\\
&=q^2I_{(q+2)q^2}+qJ_{(q+2)q^2}, 
\end{align*} 
which proves (ii). 
If $\beta+\beta'\neq0$, then 
\begin{align*}
\eqref{eq:nn}&=q\sum_{\gamma\in\mathbb{F}_{q+2}}P_{\frac{\beta\beta'}{\beta+\beta'},\gamma}\otimes C_{\varphi(\gamma)}+\sum_{\gamma,\gamma'\in\mathbb{F}_{q+2}^*,\gamma\neq \gamma'}P_{\beta\beta',\beta\gamma'+\beta'\gamma}\otimes J_{q^2}\text{ (by Lemma~\ref{lem:msl} (i))}\\ 
&=q N_{\frac{\beta\beta'}{\beta+\beta'}}+(2I_{q+2}+(q-1)J_{q+2})\otimes J_{q^2}\text{ (by Lemma~\ref{lem:msl} (iv))}\\
&=q N_{\frac{\beta\beta'}{\beta+\beta'}}+2I_{q+2}\otimes J_{q^2}+(q-1)J_{(q+2)q^2},
\end{align*} 
which proves (iii). 

For (v),  
\begin{align*}
\sum_{\beta\in\mathbb{F}_{q+2}^*}N_\beta&=\sum_{\gamma\in\mathbb{F}_{q+2}^*}(\sum_{\beta\in\mathbb{F}_{q+2}^*}P_{\beta,\gamma})\otimes C_{\varphi(\gamma)}\\
&=\sum_{\gamma\in\mathbb{F}_{q+2}^*}(J_{q+2}-I_{q+2})\otimes C_{\varphi(\gamma)}\text{ (by $\sum_{\gamma\in\mathbb{F}_{q+2}^*}P_{\beta,\gamma}=J_{q+2}-I_{q+2}$)}\\
&=(J_{q+2}-I_{q+2})\otimes \sum_{\gamma\in\mathbb{F}_{q+2}^*}C_{\varphi(\gamma)}\\
&=(J_{q+2}-I_{q+2})\otimes (q I_{q^2}+J_{q^2}). \text{ (by Lemma~\ref{lem:mfc}(i))} \qedhere
\end{align*} 

\end{proof}
Note that by Proposition~\ref{prop:1} (i), (ii) and (iii), the incidence matrices $N_\beta$ ($\beta\in\mathbb{F}_{q+2}^*$) are commuting symmetric $(q^2(q+2),q(q+1),q)$-designs.  

\subsection{An association scheme with $q+3$ classes and  its eigenmatrices}
We define the adjacency matrices as 
\begin{align*}
A_0&=I_{(q+2)q^2},\\
A_1&=(J_{q+2}-I_{q+2})\otimes I_{q^2},\\
A_2&=I_{q+2}\otimes(J_{q^2}-I_{q^2}), \\
A_{3,\beta}&=N_\beta-A_1 \quad (\beta\in\mathbb{F}_{q+2}^*).
\end{align*}

Using Proposition~\ref{prop:1}, we obtain the following theorem. 
\begin{theorem}\label{thm:as}
The matrices $A_0,A_1,A_2,A_{3,\beta}$ {\rm(}$\beta\in\mathbb{F}_{q+2}^*${\rm)} form a commutative association scheme with $q+3$ classes. 
\end{theorem}
\begin{proof}
Since $N_\beta\circ A_1=A_1$ for any $\beta\in\mathbb{F}_{q+2}^*$ where $\circ$ is the entrywise product, the adjacency matrices $A_0,A_1,A_2,A_{3,\beta}$ ($\beta\in\mathbb{F}_{q+2}^*$) are $(0,1)$-matrices. 
The condition (AS1) is obvious. 
The condition (AS2) follows from the completeness of the MSLSs, and the condition (AS3) follows from Proposition~\ref{prop:1} (i). 
Finally Proposition~\ref{prop:1} (ii), (iii), (iv) result in the conditions (AS4) and (AS5). 
\end{proof}
Note that the association scheme in Theorem~\ref{thm:as} is symmetric if and only if $q+2$ is even, namely for $q=2$. 

We further investigate the association scheme. 
Let $X=\mathbb{F}_{q+2}\times \mathbb{F}_q\times \mathbb{F}_q$. 
The binary relations on $X$ with adjacency matrices being $A_0,A_1,A_2,A_{3,\beta}$ ($\beta\in\mathbb{F}_{q+2}^*$) are given as follows:
\begin{align*}
R_0&=\{(x,x)\mid x\in X\},\\
R_1&=\{((b,a_1,a_2),(b',a_1',a_2'))\in X\times X \mid b\neq b',a_1=a_1',a_2=a_2'\},\\
R_2&=\{((b,a_1,a_2),(b',a_1',a_2'))\in X\times X \mid b= b',(a_1,a_2)\neq(a_1',a_2')\}, \\
R_{3,\beta}&=\{((\beta_1,\alpha_1,\alpha_2),(\beta_1',\alpha_{1}',\alpha_{2}') \in X\times X \mid
\beta_1 \neq \beta_1',\alpha_1\neq\alpha_1',\frac{\alpha_2-\alpha_2'}{\alpha_1-\alpha_1'}=\varphi(\beta(-\beta_1+\beta_1'))\neq y\}\\
&\quad \cup \{((\beta_1,\alpha_1,\alpha_2),(\beta_1',\alpha_{1}',\alpha_{2}') \in X\times X \mid
\alpha_1=\alpha_1',\alpha_2\neq\alpha_2', \varphi(\beta(-\beta_1+\beta_1'))= y
\}
\end{align*}
where $\beta\in\mathbb{F}_{q+2}^*$. 
It is clear that the binary relations are closed under the addition, 
that is, we have the following. 
\begin{theorem}
The association scheme is a translation scheme. 
\end{theorem}
The dual association scheme is $(X^*,\{S_0,S_1,S_2,S_{3,\beta}\mid\beta\in \mathbb{F}_{q+2}^*\})$ defined as follows: $X^*$ is the dual group of $\mathbb{F}_{q+2}\times \mathbb{F}_q\times \mathbb{F}_q$. 
Let $\chi_q$ and $\chi_{q+2}$ be the canonical additive characters of $\mathbb{F}_q$ and $\mathbb{F}_{q+2}$ respectively. 
For $\alpha_1,\alpha_2\in \mathbb{F}_q$ and $\beta\in\mathbb{F}_{q+2}$, we define a character $\chi_{\beta,\alpha_1,\alpha_2}$ of $\mathbb{F}_{q+2}\times \mathbb{F}_q\times \mathbb{F}_q$ by $\chi_{\beta',\alpha_1',\alpha_2'}(\beta,\alpha_1,\alpha_2)=\chi_{q+2}(\beta' \beta)\chi_{q}(\alpha_1'\alpha_1+\alpha_2'\alpha_2)$. 
Then 
\begin{align*}
S_0&=\{(\chi,\chi)\mid \chi\in X^*\}, \\
S_1&=\{(\chi_{\beta,\alpha_1,\alpha_2},\chi_{\beta',\alpha_1',\alpha_2'})\in X^*\times X^*\mid \beta\neq \beta',\alpha_1=\alpha_1',\alpha_2=\alpha_2'\}, \\
S_2&=\{(\chi_{\beta,\alpha_1,\alpha_2},\chi_{\beta',\alpha_1',\alpha_2'})\in X^*\times X^*\mid \beta= \beta',(\alpha_1,\alpha_2)\neq (\alpha_1',\alpha_2')\}, \\
S_{3,\beta}&=\{(\chi_{\beta_1,\alpha_1,\alpha_2},\chi_{\beta_1',\alpha_1',\alpha_2'})\in X^*\times X^* \mid
\beta_1 \neq \beta_1',
\alpha_1\neq\alpha_1',\frac{\alpha_2-\alpha_2'}{\alpha_1-\alpha_1'}=\varphi(\beta(-\beta_1+\beta_1'))\neq y
\}\\
&\quad \cup \{(\chi_{\beta_1,\alpha_1,\alpha_2},\chi_{\beta_1',\alpha_1',\alpha_2'})\in X^*\times X^* \mid
\beta_1 \neq \beta_1',
\alpha_1=\alpha_1',\alpha_2\neq\alpha_2', \varphi(\beta(-\beta_1+\beta_1'))= y)
\}
\end{align*}
where $\beta\in\mathbb{F}_{q+2}$.
By considering the bijection from $X$ to $X^*$ sending $(\beta,\alpha_1,\alpha_2)$ to $\chi_{\beta,\alpha_1,\alpha_2}$, we obtain the following result. 
\begin{theorem}
The association scheme is self-dual. 
\end{theorem}

We calculate the eigenmatrix using the additive characters of $\mathbb{F}_q$ and $\mathbb{F}_{q+2}$. 
\begin{align*}
\sum_{(\beta'',\alpha_1'',\alpha_2'')\in R_1(0,0,0)}\chi_{\beta',\alpha_1',\alpha_2'}(\beta'',\alpha_1'',\alpha_2'')&=\sum_{\beta''\in\mathbb{F}_{q+2}^*}\chi_{q+2}(\beta'\beta'')=\begin{cases}q+1 & \text{ if }\beta'=0,\\-1 & \text{ if }\beta'\neq0.\end{cases} 
\end{align*}
\begin{align*}
\sum_{(\beta'',\alpha_1'',\alpha_2'')\in R_2(0,0,0)}\chi_{\beta',\alpha_1',\alpha_2'}(\beta'',\alpha_1'',\alpha_2'')&=\sum_{\alpha_1'',\alpha_2''\in\mathbb{F}_q,(\alpha_1'',\alpha_2'')\neq(0,0)}\chi_{q}(\alpha_1'\alpha_1''+\alpha_2'\alpha_2'')\\
&=\begin{cases}
q^2-1 & \text{ if }\alpha_1'=\alpha_2'=0,\\
-1 & \text{ otherwise}.
\end{cases} 
\end{align*}

For $\beta\in\mathbb{F}_{q+2}^*$, 
\begin{align}
&\sum_{(\beta'',\alpha_1'',\alpha_2'')\in R_{3,\beta}(0,0,0)}\chi_{\beta',\alpha_1',\alpha_2'}(\beta'',\alpha_1'',\alpha_2'')\nonumber\\
&= 
\sum_{\beta''\in\mathbb{F}_{q+2}^*\setminus\{\frac{\varphi^{-1}(y)}{\beta}\},\alpha''\in\mathbb{F}_q^*}\chi_{\beta',\alpha_1',\alpha_2'}(\beta'',\alpha'',\varphi(\beta'' \beta)\alpha'')
+\sum_{\alpha''\in\mathbb{F}_q^*}\chi_{\beta',\alpha_1',\alpha_2'}(\frac{\varphi^{-1}(y)}{\beta},0,\alpha'')\nonumber\displaybreak[0]\\
&= 
\sum_{\beta''\in\mathbb{F}_{q+2}^*\setminus\{\frac{\varphi^{-1}(y)}{\beta}\}}\chi_{q+2}(\beta' \beta'')\sum_{\alpha''\in\mathbb{F}_q^*}\chi_{q}((\alpha_1'+\varphi(\beta'' \beta)\alpha_2')\alpha'')+\chi_{q+2}(\frac{\beta'\varphi^{-1}(y)}{\beta})\sum_{\alpha''\in\mathbb{F}_q^*}\chi_{q}(\alpha_2'\alpha''). \label{eq:1}
\end{align}
We now calculate \eqref{eq:1} case by case.  
\begin{enumerate}
\item In the case $\alpha_1'=\alpha_2'=0$, 
\begin{align*}
\eqref{eq:1}&=(q-1)\sum_{\beta''\in\mathbb{F}_{q+2}^*\setminus\{\frac{\varphi^{-1}(y)}{\beta}\}}\chi_{q+2}(\beta' \beta'')+(q-1)\chi_{q+2}(\frac{\beta'\varphi^{-1}(y)}{\beta})\\
&=(q-1)\sum_{\beta''\in\mathbb{F}_{q+2}^*}\chi_{q+2}(\beta' \beta'')\\
&=\begin{cases}
q^2-1  & \text{ if } \beta'=0, \\
-q+1   & \text{ if } \beta'\neq0. 
\end{cases}
\end{align*}
\item In the case $\alpha_1'\neq 0=\alpha_2'$,   
\begin{align*}
\eqref{eq:1}&=\sum_{\beta''\in\mathbb{F}_{q+2}^*\setminus\{\frac{\varphi^{-1}(y)}{\beta}\}}\chi_{q+2}(\beta' \beta'')\sum_{\alpha''\in\mathbb{F}_q^*}\chi_{q}((\alpha_1'+\varphi(\beta'' \beta)\alpha_2')\alpha'')+\chi_{q+2}(\frac{\beta'\varphi^{-1}(y)}{\beta})\sum_{\alpha''\in\mathbb{F}_q^*}\chi_{q}(\alpha_2'\alpha'')\\
&=\sum_{\beta''\in\mathbb{F}_{q+2}^*\setminus\{\frac{\varphi^{-1}(y)}{\beta}\}}\chi_{q+2}(\beta' \beta'')\sum_{\alpha''\in\mathbb{F}_q^*}\chi_{q}(\alpha_1'\alpha'')+\chi_{q+2}(\frac{\beta'\varphi^{-1}(y)}{\beta})\sum_{\alpha''\in\mathbb{F}_q^*}\chi_{q}(0)\\
&=-\sum_{\beta''\in\mathbb{F}_{q+2}^*\setminus\{\frac{\varphi^{-1}(y)}{\beta}\}}\chi_{q+2}(\beta' \beta'')+(q-1)\chi_{q+2}(\frac{\beta'\varphi^{-1}(y)}{\beta})\\
&=-\sum_{\beta''\in\mathbb{F}_{q+2}^*}\chi_{q+2}(\beta' \beta'')+q\chi_{q+2}(\frac{\beta'\varphi^{-1}(y)}{\beta})\\
&=\begin{cases}
-1  & \text{ if } \beta'=0, \\
q\chi_{q+2}(\frac{\beta'\varphi^{-1}(y)}{\beta})+1   & \text{ if } \beta'\neq0. 
\end{cases}
\end{align*} 
\item In the case $\alpha_2'\neq 0$ there uniquely exists $\bar{\beta}\in\mathbb{F}_{q+2}$ such that $\alpha_1'+\varphi(\beta\bar{\beta})\alpha_2'=0$. Then   
\begin{align*}
\eqref{eq:1}&=\sum_{\beta''\in\mathbb{F}_{q+2}^*\setminus\{\frac{\varphi^{-1}(y)}{\beta}\}}\chi_{q+2}(\beta' \beta'')\sum_{\alpha''\in\mathbb{F}_q^*}\chi_{q}((\alpha_1'+\varphi(\beta'' \beta)\alpha_2')\alpha'')+\chi_{q+2}(\frac{\beta'\varphi^{-1}(y)}{\beta})\sum_{\alpha''\in\mathbb{F}_q^*}\chi_{q}(\alpha_2'\alpha'')\\
&=(q-1)\chi_{q+2}(\beta' \bar{\beta})-\sum_{\beta''\in\mathbb{F}_{q+2}^*\setminus\{\frac{\varphi^{-1}(y)}{\beta},\bar{\beta}\}}\chi_{q+2}(\beta' \beta'')-\chi_{q+2}(\frac{\beta'\varphi^{-1}(y)}{\beta})\\
&=q\chi_{q+2}(\beta' \bar{\beta})-\sum_{\beta''\in\mathbb{F}_{q+2}^*}\chi_{q+2}(\beta' \beta'')\\
&=\begin{cases}
-1  & \text{ if } \beta'=0, \\
q\chi_{q+2}(\beta' \bar{\beta})+1   & \text{ if } \beta'\neq0. 
\end{cases}
\end{align*}
\end{enumerate}

Let $V_{i}$ be as follows
\begin{align*}
V_{0}&=\text{span}_\mathbb{C}\{\chi_{0,0,0}\},\\
V_{1}&=\text{span}_\mathbb{C}\{\chi_{\beta,0,0}\mid \beta\in\mathbb{F}_{q+2}^*\},\\
V_{2}&=\text{span}_\mathbb{C}\{\chi_{0,\alpha_1,\alpha_2}\mid \alpha_1,\alpha_2\in\mathbb{F}_{q}^*,(\alpha_1,\alpha_2)\neq(0,0)\},\\
V_{3,\tilde{\beta}}
&=\text{span}_\mathbb{C}\{\chi_{\beta',\alpha_1',\alpha_2'}\mid \beta'\in\mathbb{F}_{q+2}^*,\alpha_1'\in\mathbb{F}_q,\alpha_2'\in\mathbb{F}_q^*,\beta'\varphi^{-1}(-\frac{\alpha_1'}{\alpha_2'})=\tilde{\beta}\} \\
&\quad +\text{span}_\mathbb{C}\{\chi_{\beta',\alpha_1',0}\mid \beta'\in\mathbb{F}_{q+2}^*,\alpha_1'\in\mathbb{F}_q^*,\beta'\varphi^{-1}(y)=\tilde{\beta}\}, 
\end{align*}
where $\tilde{\beta}\in\mathbb{F}_{q+2}^*$. 
From the above calculation, $V_i$'s are maximal common eigenspaces of $A_0,A_1,A_2,A_{3,\beta}$ ($\beta\in\mathbb{F}_{q+2}^*$). 
Thus we obtain the following formula for the eigenmatrix.
\begin{theorem}
The first eigenmatrix $P$ of the association scheme is


\begin{align*}
P=\bordermatrix{
      & R_0 & R_1 & R_2  & R_{3,\beta}  \cr
V_0 & 1 &  q+1 & q^2-1 & q^2-1  \cr
V_1 & 1 &  -1  & q^2-1  & -q+1  \cr
V_2 & 1 &  q+1 & -1  & -1  \cr
V_{3,\tilde{\beta}} & 1 &  -1 & -1  & q \chi_{q+2}(\frac{\tilde{\beta}}{\beta})+1
},
\end{align*}
where $\beta,\tilde{\beta}$ run over the set $\mathbb{F}_{q+2}^*$. 

\end{theorem}

\begin{example}
We describe the construction for twin primes $3,5$.  

Let $\phi:\mathbb{F}_3=\{0,1,2\}\rightarrow GL_3(\mathbb{R});\phi(x)=(r_3)^x$ where $r_3=\left(\begin{smallmatrix}0 & 1 & 0 \\ 0 & 0 & 1 \\ 1 & 0 & 0 \end{smallmatrix}\right)$. 
The generalized Hadamard matrix is $H_3=\left(\begin{smallmatrix}0 & 0 & 0 \\ 0 & 1 & 2 \\ 0 & 2 & 1 \end{smallmatrix}\right)$. 
We construct three auxiliary matrices $C_0,C_1,C_2$ from $\mathbb{F}_3$;  
\begin{align*}
C_0=\begin{pmatrix}
\phi(0) & \phi(0) & \phi(0) \\
\phi(0) & \phi(0) & \phi(0) \\
\phi(0) & \phi(0) & \phi(0) 
\end{pmatrix},C_1=\begin{pmatrix}
\phi(0) & \phi(1) & \phi(2) \\
\phi(2) & \phi(0) & \phi(1) \\
\phi(1) & \phi(2) & \phi(0) 
\end{pmatrix},C_2=\begin{pmatrix}
\phi(0) & \phi(2) & \phi(1) \\
\phi(1) & \phi(0) & \phi(2) \\
\phi(2) & \phi(1) & \phi(0) 
\end{pmatrix}. 
\end{align*}
Furthermore, we let $C_x=O_9$ and $C_y=I_3\otimes J_3$ where $x,y$ are indeterminates. 

We construct four Latin squares $L_1,L_2,L_3,L_4$ from $\mathbb{F}_5=\{0,1,2,3,4\}$ which are mutually suitable Latin squares with constant diagonal entries. 
\begin{align*}
L_1&=\left(
\begin{array}{ccccc}
 0 & 1 & 2 & 3 & 4 \\
 4 & 0 & 1 & 2 & 3 \\
 3 & 4 & 0 & 1 & 2 \\
 2 & 3 & 4 & 0 & 1 \\
 1 & 2 & 3 & 4 & 0 \\
\end{array}
\right),\quad 
L_2=
\left(
\begin{array}{ccccc}
 0 & 2 & 4 & 1 & 3 \\
 3 & 0 & 2 & 4 & 1 \\
 1 & 3 & 0 & 2 & 4 \\
 4 & 1 & 3 & 0 & 2 \\
 2 & 4 & 1 & 3 & 0 \\
\end{array}
\right),\\
L_3&=\left(
\begin{array}{ccccc}
 0 & 3 & 1 & 4 & 2 \\
 2 & 0 & 3 & 1 & 4 \\
 4 & 2 & 0 & 3 & 1 \\
 1 & 4 & 2 & 0 & 3 \\
 3 & 1 & 4 & 2 & 0 \\
\end{array}
\right),\quad 
L_4=\left(
\begin{array}{ccccc}
 0 & 4 & 3 & 2 & 1 \\
 1 & 0 & 4 & 3 & 2 \\
 2 & 1 & 0 & 4 & 3 \\
 3 & 2 & 1 & 0 & 4 \\
 4 & 3 & 2 & 1 & 0 \\
\end{array}
\right).
\end{align*}

Fix a bijection $\varphi: \mathbb{F}_5\rightarrow \mathbb{F}_3\cup\{x,y\}$ such that $\varphi(0)=x$. 
We now define the incidence matrices of symmetric $(45,12,3)$-designs $N_i$ ($i\in\mathbb{F}_5^*$) by replacing $j\in\mathbb{F}_5$ in $L_i$ with $C_{\varphi(j)}$.
For example,  
\begin{align*}
N_1&=\left(
\begin{array}{ccccc}
 C_x & C_{\varphi(1)} & C_{\varphi(2)} & C_{\varphi(3)} & C_{\varphi(4)} \\
 C_{\varphi(4)} & C_x & C_{\varphi(1)} & C_{\varphi(2)} & C_{\varphi(3)} \\
 C_{\varphi(3)} & C_{\varphi(4)} & C_x & C_{\varphi(1)} & C_{\varphi(2)} \\
 C_{\varphi(2)} & C_{\varphi(3)} & C_{\varphi(4)} & C_x & C_{\varphi(1)} \\
 C_{\varphi(1)} & C_{\varphi(2)} & C_{\varphi(3)} & C_{\varphi(4)} & C_x \\
\end{array}
\right).
\end{align*}

Then the matrices $I_{45}, (J_{5}-I_{5})\otimes I_{9},I_{5}\otimes (J_{9}-I_{9})$, $N_i-(J_5-I_5)\otimes I_9$ ($i\in \mathbb{F}_5^*$)
 form a commutative association scheme with $6$ classes. 
The first eigenmatrix $P$ is 
\begin{align*}
P=\bordermatrix{
 & R_0 & R_1 & R_2 & R_{3,1} & R_{3,2} & R_{3,3} & R_{3,4} \cr
V_0 & 1 & 4 & 8 & 8 & 8 & 8 & 8  \cr
V_1 & 1 & -1& 8 & -2 & -2& -2& -2 \cr
V_2 & 1 & 4 & -1 & -1 & -1& -1& -1 \cr
V_{3,1} & 1 & -1 & -1 & 3w+1 & 3w^2+1& 3w^3+1& 3w^4+1  \cr
V_{3,2} & 1 & -1 & -1 & 3w^3+1 & 3w+1& 3w^4+1& 3w^2+1  \cr
V_{3,3} & 1 & -1 & -1 & 3w^2+1 & 3w^4+1& 3w+1& 3w^3+1  \cr
V_{3,4} & 1 & -1 & -1 & 3w^4+1 & 3w^3+1& 3w^2+1& 3w+1  
},
\end{align*}
where $w=\frac{\sqrt{5}-1}{4}+\sqrt{\frac{-5-\sqrt{5}}{8}}$. 
\end{example}

\section{Association schemes obtained from Merssene primes and Fermat primes}\label{sec:asmfp}
In this section we use prime powers $q$ and $q+1$ to construct a set of symmetric group divisible designs and derive a commutative association scheme from it. 
\subsection{Symmetric group divisible designs}
Let $q,q+1$ be prime powers. 
Fix a bijection $\varphi:\mathbb{F}_{q+1}\rightarrow \mathbb{F}_q\cup\{x\}$ such that $\varphi(0)=x$. 
Consider a Latin square obtained from $L_\beta$ by replacing entries with their image of $\varphi$, which we denote by $L_{\varphi(\beta)}$. 
Recall that we denote the $(u,v)$-entry of an array $L$ by $L(u,v)$. 
Then for $\beta,\beta',\beta''\in\mathbb{F}_{q+1}$, $L_{\varphi(\beta)}(\beta',\beta'')=\varphi(L_{\beta}(\beta',\beta''))$. 

We now construct symmetric group divisible designs from auxiliary matrices for $\mathbb{F}_q$ and mutually suitable Latin squares for $\mathbb{F}_{q+1}$.
For $\beta\in\mathbb{F}_{q+1}^*$, we define a $(q+1)q^2\times (q+1)q^2$ $(0,1)$-matrix $N_\beta$ to be a $(q+1)\times (q+1)$ block matrix with rows and columns indexed by $\mathbb{F}_{q+1}$ whose $(\beta',\beta'')$-block matrix is $C_{L_{\varphi(\beta)}(\beta',\beta'')}$;  
\begin{align*}
N_\beta=(C_{L_{\varphi(\beta)}(\beta',\beta'')})_{\beta',\beta''\in\mathbb{F}_{q+1}}=\sum_{\gamma\in\mathbb{F}_{q+1}}P_{\beta,\gamma}\otimes C_{\varphi(\gamma)}. 
\end{align*}

\begin{proposition}\label{prop:mf1}
\begin{enumerate}
\item For any $\beta\in\mathbb{F}_{q+1}^*$, $N_\beta^\top=N_{-\beta}$.   
\item For any $\beta\in\mathbb{F}_{q+1}^*$, $N_{\beta}N_{-\beta}=q^2I_{(q+1)q^2}-q I_{(q+1)q}\otimes J_q+I_{q+1}\otimes J_{q^2}+(q-1)J_{(q+1)q^2}$. 
\item For any $\beta,\beta'\in\mathbb{F}_{q+1}^*$ such that $\beta+\beta'\neq 0$, $N_{\beta} N_{\beta'}=q N_{\frac{\beta\beta'}{\beta+\beta'}}+2I_{q+1}\otimes J_{q^2}+(q-2)J_{(q+1)q^2}$. 
\item For any $\beta\in\mathbb{F}_{q+1}^*$, $N_\beta(I_{(q+1)q}\otimes J_{q})=(I_{(q+1)q}\otimes J_{q})N_\beta=(J_{(q+1)q^2}-I_{q+1}\otimes J_{q^2})$.
\item For any $\beta\in\mathbb{F}_{q+1}^*$, $N_\beta(I_{q+1}\otimes J_{q^2})=(I_{q+1}\otimes J_{q^2})N_\beta=q(J_{(q+1)q^2}-I_{q+1}\otimes J_{q^2})$. 
\item $\sum_{\beta\in\mathbb{F}_{q+1}^*}N_\beta=(J_{q+1}-I_{q+1})\otimes (q I_{q^2}+(J_{q}-I_q)\otimes J_q)$. 
\end{enumerate}
\end{proposition}
\begin{proof}
The proof is similar to the proof of Proposition~\ref{prop:1}, but for the sake of completeness we include a proof.
(i) is easy to see, and (iv), (v)  follow from Lemma~\ref{lem:mfc}(iv), (v) respectively.  
We prove (ii) as well as (iii) in a same manner as in Proposition~\ref{prop:1}.  
For $\beta,\beta'\in\mathbb{F}_{q+1}^*$, by Lemma~\ref{lem:mfc} (ii), (iii),
\begin{align}
N_\beta N_{\beta'}&=\sum_{\gamma,\gamma'\in\mathbb{F}_{q+1}}P_{\beta\beta',\beta\gamma'+\beta'\gamma}\otimes C_{\varphi(\gamma)}C_{\varphi(\gamma')}\nonumber\\
&=\sum_{\gamma\in\mathbb{F}_{q+1}}P_{\beta\beta',(\beta+\beta')\gamma}\otimes (C_{\varphi(\gamma)})^2+\sum_{\gamma,\gamma'\in\mathbb{F}_{q+1},\gamma\neq \gamma'}P_{\beta\beta',\beta\gamma'+\beta'\gamma}\otimes C_{\varphi(\gamma)}C_{\varphi(\gamma')}\nonumber\\
&=q\sum_{\gamma\in\mathbb{F}_{q+1}}P_{\beta\beta',(\beta+\beta')\gamma}\otimes C_{\varphi(\gamma)}+\sum_{\gamma,\gamma'\in\mathbb{F}_{q+1}^*,\gamma\neq \gamma'}P_{\beta\beta',\beta\gamma'+\beta'\gamma}\otimes J_{q^2}. \label{eq:mfnn} 
\end{align}
\eqref{eq:mfnn} is calculated depending on whether $\beta+\beta'=0$ or not as follows.
If $\beta+\beta'=0$, then 
\begin{align*}
\eqref{eq:mfnn}&=q\sum_{\gamma\in\mathbb{F}_{q+1}}P_{\beta\beta',0}\otimes C_{\varphi(\gamma)}+\sum_{\gamma,\gamma'\in\mathbb{F}_{q+1}^*,\gamma\neq \gamma'}P_{\beta\beta,\beta\gamma'+\beta'\gamma}\otimes J_{q^2}\\ 
&=qI_{q+1}\otimes \sum_{\gamma\in\mathbb{F}_{q+1}}C_{\varphi(\gamma)}+(q-1)(J_{q+1}-I_{q+1})\otimes J_{q^2}\text{ (by Lemma~\ref{lem:msl} (iv))}\\
&=qI_{q+1}\otimes (qI_{q^2}+(J_{q}-I_q)\otimes J_q)+(q-1)(J_{q+1}-I_{q+1})\otimes J_{q^2}\text{ (by Lemma~\ref{lem:mfc} (iv))}\\
&=q^2I_{(q+1)q^2}-q I_{(q+1)q}\otimes J_q+I_{q+1}\otimes J_{q^2}+(q-1)J_{(q+1)q^2}, 
\end{align*} 
which proves (ii). 
If $\beta+\beta'\neq0$, then 
\begin{align*}
\eqref{eq:mfnn}&=q\sum_{\gamma\in\mathbb{F}_{q+1}}P_{\frac{\beta\beta'}{\beta+\beta'},\gamma}\otimes C_{\varphi(\gamma)}+\sum_{\gamma,\gamma'\in\mathbb{F}_{q+1}^*,\gamma\neq \gamma'}P_{\beta\beta,\beta\gamma'+\beta'\gamma}\otimes J_{q^2}\text{ (by Lemma~\ref{lem:msl} (i))}\\ 
&=q N_{\frac{\beta\beta'}{\beta+\beta'}}+(2I_{q+1}+(q-2)J_{q+1})\otimes J_{q^2}\text{ (by Lemma~\ref{lem:msl} (iv))}\\
&=q N_{\frac{\beta\beta'}{\beta+\beta'}}+2I_{q+1}\otimes J_{q^2}+(q-2)J_{(q+1)q^2},
\end{align*} 
which proves (iii). 

For (vi),  
\begin{align*}
\sum_{\beta\in\mathbb{F}_{q+1}^*}N_\beta&=\sum_{\gamma\in\mathbb{F}_{q+1}^*}(\sum_{\beta\in\mathbb{F}_{q+1}^*}P_{\beta,\gamma})\otimes C_{\varphi(\gamma)}\\
&=\sum_{\gamma\in\mathbb{F}_{q+1}^*}(J_{q+1}-I_{q+1})\otimes C_{\varphi(\gamma)}\text{ (by $\sum_{\gamma\in\mathbb{F}_{q+1}^*}P_{\beta,\gamma}=J_{q+1}-I_{q+1}$)}\\
&=(J_{q+1}-I_{q+1})\otimes \sum_{\gamma\in\mathbb{F}_{q+1}^*}C_{\varphi(\gamma)}\\
&=(J_{q+1}-I_{q+1})\otimes (q I_{q^2}+(J_{q}-I_q)\otimes J_q). \text{ (by Lemma~\ref{lem:mfc}(i))} \qedhere
\end{align*} 
\end{proof}

\begin{corollary}
For any $\beta\in\mathbb{F}_{q+1}^*$, 
\begin{align*}
&(N_\beta+I_{q+1}\otimes (J_{q^2}-I_q \otimes J_q))(N_\beta^\top+I_{q+1}\otimes (J_{q^2}-I_q \otimes J_q))\\
&=q^2I_{(q+1)q^2}+(q-1)(q-3)I_{q+1}\otimes J_{q^2}+3(q-1)J_{(q+1)q^2},  
\end{align*}
that is, $N_\beta+I_{q+1}\otimes (J_{q^2}-I_q \otimes J_q)$ is the incidence matrix of a  symmetric group divisible design with parameters $((q+1)q^2,2q^2-q,q+1,q^2,q(q-1),3(q-1))$.  
\end{corollary}
Note that by Proposition~\ref{prop:mf1}(iii), the incidence matrices $N_\beta$ ($\beta\in\mathbb{F}_{q+1}^*$) are commuting.  

\subsection{An association scheme with $q+4$ classes and  its eigenmatrices}
We define the adjacency matrices as 
\begin{align*}
A_0&=I_{(q+1)q^2},\\
A_1&=I_{q+1}\otimes I_{q}\otimes (J_q-I_q),\\
A_2&=I_{q+1}\otimes (J_q-I_q)\otimes J_q, \\
A_3&=(J_{q+1}-I_{q+1})\otimes I_{q^2},\\
A_4&=(J_{q+1}-I_{q+1})\otimes I_{q}\otimes (J_q-I_q), \\
A_{5,\beta}&=N_\beta-A_3 \quad (\beta\in\mathbb{F}_{q+1}^*).
\end{align*}

Using Proposition~\ref{prop:mf1}, we obtain the following theorem. 
\begin{theorem}\label{thm:mfas}
The matrices $A_0,A_1,A_2,A_3,A_4,A_{5,\beta}$ {\rm(}$\beta\in\mathbb{F}_{q+1}^*${\rm)} form a commutative association scheme with $q+4$ classes. 
\end{theorem}
\begin{proof}
It is easy to see that the conditions (AS1), (AS2), (AS3) hold. 

We check (AS4) case by case. Let $\mathcal{A}$ be the vector space over the complex number field spanned by $A_i,A_{5,\beta}$ ($i\in\{0,1,2,3,4\},\beta\in\mathbb{F}_{q+1}^*$). 
For $i,j\in\{0,1,2,3,4\}$, it is trivial that $A_iA_j\in\mathcal{A}$.   
For $i\in\{1,2\}$ and $\beta\in\mathbb{F}_{q+1}$, $A_i A_{5,\beta}\in\mathcal{A}$ holds by Lemma~\ref{lem:mfc} (iv), (v).   
From Lemma~\ref{lem:mfc} and the fact that $P_{\beta,\gamma}J_{q^2}=J_{q^2}P_{\beta,\gamma}=J_{q^2}$, it follows that ($J_{q^2}\otimes I_{q^2})N_\beta$ and $(J_{q+1}\otimes I_q\otimes J_q)N_\beta$ are both in $\mathcal{A}$, and therefore we have that $A_i A_{5,\beta}\in\mathcal{A}$ for $i\in\{3,4\}$ and $\beta\in\mathbb{F}_{q+1}$.  
Thus $\mathcal{A}$ is closed under the ordinary matrix multiplication. 
Finally (AS5) follows from Proposition~\ref{prop:mf1}(iii), (iv), (v). 
\end{proof}
Note that the associations scheme in Theorem~\ref{thm:mfas} is symmetric if and only if $q+1$ is even. 

We further investigate the association scheme. 
Let $X=\mathbb{F}_{q+1}\times \mathbb{F}_q\times \mathbb{F}_q$. 
The binary relations on $X$ with adjacency matrices being $A_0,A_1,A_2,A_3,A_4,A_{5,\beta}$ ($\beta\in\mathbb{F}_{q+1}^*$) are given as follows:
\begin{align*}
R_0&=\{(x,x)\mid x\in X\},\\
R_1&=\{((b,a_1,a_2),(b',a_1',a_2'))\in X\times X \mid b= b',a_1=a_1',a_2\neq a_2'\},\\
R_2&=\{((b,a_1,a_2),(b',a_1',a_2'))\in X\times X \mid b= b',a_1\neq a_1'\}, \\
R_3&=\{((b,a_1,a_2),(b',a_1',a_2'))\in X\times X \mid b\neq b',a_1=a_1',a_2=a_2'\},\\
R_4&=\{((b,a_1,a_2),(b',a_1',a_2'))\in X\times X \mid b\neq  b',a_1= a_1',a_2\neq a_2'\}, \\
R_{5,\beta}&=\{((\beta_1,\alpha_1,\alpha_2),(\beta_1',\alpha_{1}',\alpha_{2}') \in X\times X \mid
\beta_1 \neq \beta_1',\alpha_1\neq\alpha_1',\frac{\alpha_2-\alpha_2'}{\alpha_1-\alpha_1'}=\varphi(\beta(-\beta_1+\beta_1'))\}
\end{align*}
where $\beta\in\mathbb{F}_{q+1}^*$. 
It is clear that the binary relations are closed under the addition, 
that is, we have the following. 
\begin{theorem}
The association scheme is a translation scheme. 
\end{theorem}
The dual association scheme is $(X^*,\{S_0,S_1,S_2,S_3,S_4,S_{5,\beta}\mid\beta\in \mathbb{F}_{q+1}^*\})$ defined as follows: $X^*$ is the dual group of the additive group $\mathbb{F}_{q+1}\times \mathbb{F}_q\times \mathbb{F}_q$. 
Let $\chi_q$ and $\chi_{q+1}$ be the canonical additive characters of $\mathbb{F}_q$ and $\mathbb{F}_{q+1}$ respectively. 
For $\alpha_1,\alpha_2\in \mathbb{F}_q$ and $\beta\in\mathbb{F}_{q+1}$, we define a character $\chi_{\beta,\alpha_1,\alpha_2}$ of $\mathbb{F}_{q+1}\times \mathbb{F}_q\times \mathbb{F}_q$ by $\chi_{\beta',\alpha_1',\alpha_2'}(\beta,\alpha_1,\alpha_2)=\chi_{q+1}(\beta' \beta)\chi_{q}(\alpha_1'\alpha+\alpha_2'\alpha_2)$. 
Then 
\begin{align*}
S_0&=\{(\chi,\chi)\mid \chi\in X^*\}, \\
S_1&=\{(\chi_{\beta,\alpha_1,\alpha_2},\chi_{\beta',\alpha_1',\alpha_2'})\in X^*\times X^*\mid \beta= \beta',\alpha_1=\alpha_1',\alpha_2\neq \alpha_2'\}, \\
S_2&=\{(\chi_{\beta,\alpha_1,\alpha_2},\chi_{\beta',\alpha_1',\alpha_2'})\in X^*\times X^*\mid \beta= \beta',\alpha_1\neq \alpha_1'\}, \\
S_3&=\{(\chi_{\beta,\alpha_1,\alpha_2},\chi_{\beta',\alpha_1',\alpha_2'})\in X^*\times X^*\mid \beta\neq \beta',\alpha_1=\alpha_1',\alpha_2=\alpha_2'\}, \\
S_4&=\{(\chi_{\beta,\alpha_1,\alpha_2},\chi_{\beta',\alpha_1',\alpha_2'})\in X^*\times X^*\mid \beta\neq \beta',\alpha_1=\alpha_1',\alpha_2\neq \alpha_2'\}, \\
S_{5,\beta}&=\{(\chi_{\beta_1,\alpha_1,\alpha_2},\chi_{\beta_1',\alpha_1',\alpha_2'})\in X^*\times X^* \mid
\beta_1 \neq \beta_1',
\alpha_1\neq\alpha_1',\frac{\alpha_2-\alpha_2'}{\alpha_1-\alpha_1'}=\varphi(\beta(-\beta_1+\beta_1'))
\}
\end{align*}
where $\beta\in\mathbb{F}_{q+1}$.
By considering the bijection from $X$ to $X^*$ sending $(\beta,\alpha_1,\alpha_2)$ to $\chi_{\beta,\alpha_1,\alpha_2}$, we obtain the following result. 
\begin{theorem}
The association scheme is self-dual. 
\end{theorem}

We calculate the eigenmatrix using the additive characters of $\mathbb{F}_q$ and $\mathbb{F}_{q+1}$.
\begin{align*}
\sum_{(\beta'',\alpha_1'',\alpha_2'')\in R_1(0,0,0)}\chi_{\beta',\alpha_1',\alpha_2'}(\beta'',\alpha_1'',\alpha_2'')&=\sum_{\alpha_2''\in\mathbb{F}_q^*}\chi_{q}(\alpha_2'\alpha_2'')
=\begin{cases}q-1 & \text{ if }\alpha_2'=0,\\-1 & \text{ if }\alpha_2'\neq0.\end{cases} 
\end{align*}

\begin{align*}
\sum_{(\beta'',\alpha_1'',\alpha_2'')\in R_2(0,0,0)}\chi_{\beta',\alpha_1',\alpha_2'}(\beta'',\alpha_1'',\alpha_2'')&
=(\sum_{\alpha_1''\in\mathbb{F}_q^*}\chi_{q}(\alpha_1'\alpha_1''))(\sum_{\alpha_2''\in\mathbb{F}_q}\chi_{q}(\alpha_2'\alpha_2''))\\
&=\begin{cases}
(q-1)q & \text{ if }\alpha_1'=0,\alpha_2'=0,\\
-q & \text{ if }\alpha_1'\neq0,\alpha_2'=0,\\
0 & \text{ if }\alpha_2'\neq 0.\end{cases} 
\end{align*}

\begin{align*}
\sum_{(\beta'',\alpha_1'',\alpha_2'')\in R_3(0,0,0)}\chi_{\beta',\alpha_1',\alpha_2'}(\beta'',\alpha_1'',\alpha_2'')&=\sum_{\beta''\in\mathbb{F}_{q+1}^*}\chi_{q+1}(\beta'\beta'')=\begin{cases}q & \text{ if }\beta'=0,\\-1 & \text{ if }\beta'\neq0.\end{cases} 
\end{align*}

\begin{align*}
\sum_{(\beta'',\alpha_1'',\alpha_2'')\in R_4(0,0,0)}\chi_{\beta',\alpha_1',\alpha_2'}(\beta'',\alpha_1'',\alpha_2'')&=(\sum_{\beta''\in\mathbb{F}_{q+1}^*}\chi_{q+1}(\beta'\beta''))(\sum_{\alpha_2''\in\mathbb{F}_q^*}\chi_{q}(\alpha_2'\alpha_2''))\\
&=\begin{cases}(q-1)q & \text{ if }\beta'=0,\alpha_2'=0,\\-q & \text{ if }\beta'\neq0,\alpha_2'=0,\\ -q+1 & \text{ if }\beta'=0,\alpha_2'\neq 0,\\ 1 & \text{ if }\beta'\neq0,\alpha_2'\neq 0.\end{cases} 
\end{align*}

For $\beta\in\mathbb{F}_{q+1}^*$, 
\begin{align}
&\sum_{(\beta'',\alpha_1'',\alpha_2'')\in R_{5,\beta}(0,0,0)}\chi_{\beta',\alpha_1',\alpha_2'}(\beta'',\alpha_1'',\alpha_2'')\nonumber\\
&= 
\sum_{\beta''\in\mathbb{F}_{q+2}^*,\alpha''\in\mathbb{F}_q^*}\chi_{\beta',\alpha_1',\alpha_2'}(\beta'',\alpha'',\varphi(\beta'' \beta)\alpha'')
\nonumber\displaybreak[0]\\
&= 
\sum_{\beta''\in\mathbb{F}_{q+2}^*}\chi_{q+2}(\beta' \beta'')\sum_{\alpha''\in\mathbb{F}_q^*}\chi_{q}((\alpha_1'+\varphi(\beta'' \beta)\alpha_2')\alpha''). \label{eq:mf1}
\end{align}

We now calculate \eqref{eq:mf1} case by case.  
\begin{enumerate}
\item In the case $\alpha_1'=\alpha_2'=0$, 
\begin{align*}
\eqref{eq:mf1}&=(q-1)\sum_{\beta''\in\mathbb{F}_{q+1}^*}\chi_{q+1}(\beta' \beta'')\\
&=\begin{cases}
q(q-1)  & \text{ if } \beta'=0, \\
-q+1   & \text{ if } \beta'\neq0. 
\end{cases}
\end{align*}
\item In the case $\alpha_1'\neq 0=\alpha_2'$,   
\begin{align*}
\eqref{eq:mf1}&=\sum_{\beta''\in\mathbb{F}_{q+1}^*}\chi_{q+1}(\beta' \beta'')\sum_{\alpha''\in\mathbb{F}_q^*}\chi_{q}(\alpha_1'\alpha'')\\
&=-\sum_{\beta''\in\mathbb{F}_{q+1}^*}\chi_{q+1}(\beta' \beta'')\\
&=\begin{cases}
-q  & \text{ if } \beta'=0, \\
1   & \text{ if } \beta'\neq0. 
\end{cases}
\end{align*} 

\item In the case $\alpha_2'\neq 0$ there uniquely exists $\bar{\beta}\in\mathbb{F}_{q+1}$ such that $\alpha_1'+\varphi(\beta\bar{\beta})\alpha_2'=0$. Then   
\begin{align*}
\eqref{eq:mf1}&= 
\sum_{\beta''\in\mathbb{F}_{q+2}^*}\chi_{q+1}(\beta' \beta'')\sum_{\alpha''\in\mathbb{F}_q^*}\chi_{q}((\alpha_1'+\varphi(\beta'' \beta)\alpha_2')\alpha'')\\
&=(q-1)\chi_{q+1}(\beta' \bar{\beta})-\sum_{\beta''\in\mathbb{F}_{q+1}^*\setminus\{\bar{\beta}\}}\chi_{q+1}(\beta' \beta'')\\
&=q\chi_{q+1}(\beta' \bar{\beta})-\sum_{\beta''\in\mathbb{F}_{q+1}^*}\chi_{q+1}(\beta' \beta'')\\
&=\begin{cases}
0  & \text{ if } \beta'=0, \\
q\chi_{q+1}(\beta' \bar{\beta})+1   & \text{ if } \beta'\neq0. 
\end{cases}
\end{align*}
\end{enumerate}

Let $V_{i}$ be as follows
\begin{align*}
V_{0}&=\text{span}_\mathbb{C}\{\chi_{0,0,0}\},\\
V_{1}&=\text{span}_\mathbb{C}\{\chi_{0,\alpha_1,0}\mid \alpha_1\in\mathbb{F}_{q}^*\},\\
V_{2}&=\text{span}_\mathbb{C}\{\chi_{0,\alpha_1,\alpha_2}\mid \alpha_1\in\mathbb{F}_q,\alpha_2\in\mathbb{F}_{q}^*\},\\
V_{3}&=\text{span}_\mathbb{C}\{\chi_{\beta,0,0}\mid \beta\in\mathbb{F}_{q+2}^*\},\\
V_{4}&=\text{span}_\mathbb{C}\{\chi_{\beta,\alpha_1,0}\mid \beta\in\mathbb{F}_{q+2}^*,\alpha_1\in\mathbb{F}_{q}^*\},\\
V_{5,\tilde{\beta}}
&=\text{span}_\mathbb{C}\{\chi_{\beta',\alpha_1',\alpha_2'}\mid \beta'\in\mathbb{F}_{q+1}^*,\alpha_1',\alpha_2'\in\mathbb{F}_q^*,\beta'\varphi^{-1}(-\frac{\alpha_1'}{\alpha_2'})=\tilde{\beta}\}, 
\end{align*}
where $\tilde{\beta}\in\mathbb{F}_{q+1}^*$. 
From the above calculation, $V_i$'s are maximal common eigenspaces of $A_0,A_1,A_2,A_{3,\beta}$ ($\beta\in\mathbb{F}_{q+1}^*$). 
Thus we obtain the following formula for the eigenmatrix.
\begin{theorem}
The first eigenmatrix $P$ of the association scheme is


\begin{align*}
P=\bordermatrix{
      & R_0 & R_1 & R_2  & R_3 & R_4  & R_{5,\beta}  \cr
V_0 & 1 &  q-1 & q(q-1) & q & q(q-1) & q(q-1) \cr
V_1 & 1 &  q-1 & -q & q & q(q-1) & -q \cr
V_2 & 1 &  -1 & 0 & q & -q+1 & 0 \cr
V_3 & 1 &  q-1 & q(q-1)  & -1 & -q & -q+1 \cr
V_4 & 1 &  -1 & 0 & -1 & 1 & 1 \cr
V_{5,\tilde{\beta}} & 1 &  -1 & 0  & -1 & 1 &q\chi_{q+1}(\frac{\bar{\beta}}{\beta})+1
},
\end{align*}
where $\beta,\tilde{\beta}$ run over the set $\mathbb{F}_{q+1}^*$. 

\end{theorem}

\begin{example}

We describe the construction for prime powers $4,5$.  

Let $\mathbb{F}_4=\{0,1,z,z+1\}$ with $z^2=z+1$ be the finite field of order $4$. 
We regard $\mathbb{F}_4$ as $\mathbb{Z}_2^2$ as the additive group, and  
let $\phi:\mathbb{F}_4\rightarrow GL_4(\mathbb{R});\phi(a+bz)=(r_2)^{a}\otimes (r_2)^{b}$ where $r_2=\left(\begin{smallmatrix}0 & 1 \\ 1 & 0 \end{smallmatrix}\right)$. 
The generalized Hadamard matrix is $H_4=\left(\begin{smallmatrix}0 & 0 & 0 & 0 \\ 0 & 1 & z & z+1 \\ 0 & z & z+1 & 1 \\ 0 & z+1 & 1 & z \end{smallmatrix}\right)$. 
We construct four auxiliary matrices $C_0,C_1,C_z,C_{z+1}$ from $H_4$;  
\begin{align*}
C_0&=\begin{pmatrix}
\phi(0) & \phi(0) & \phi(0) & \phi(0) \\
\phi(0) & \phi(0) & \phi(0) & \phi(0) \\
\phi(0) & \phi(0) & \phi(0) & \phi(0) \\
\phi(0) & \phi(0) & \phi(0) & \phi(0) 
\end{pmatrix},C_1=\begin{pmatrix}
\phi(0) & \phi(1) & \phi(z) & \phi(z+1) \\
\phi(1) & \phi(0) & \phi(z+1) & \phi(z) \\
\phi(z) & \phi(z+1) & \phi(0) & \phi(1) \\ 
\phi(z+1) & \phi(z) & \phi(1) & \phi(0) 
\end{pmatrix},\\
C_z&=\begin{pmatrix}
\phi(0) & \phi(z) & \phi(z+1) & \phi(1) \\
\phi(z) & \phi(0) & \phi(1) & \phi(z+1) \\
\phi(z+1) & \phi(1) & \phi(0) & \phi(z) \\ 
\phi(1) & \phi(z+1) & \phi(z) & \phi(0) 
\end{pmatrix}, 
C_{z+1}=\begin{pmatrix}
\phi(0) & \phi(z+1) & \phi(1) & \phi(z) \\
\phi(z+1) & \phi(0) & \phi(z) & \phi(1) \\
\phi(1) & \phi(z) & \phi(0) & \phi(z+1) \\ 
\phi(z) & \phi(1) & \phi(z+1) & \phi(0) 
\end{pmatrix}. 
\end{align*}
Furthermore, we let $C_x=O_{16}$ where $x$ is an indeterminate. 

Let $\mathbb{F}_5=\{0,1,2,3,4\}$ be the finite field of order $5$. 
We construct four Latin squares $L_1,L_2,L_3,L_4$ from $\mathbb{F}_5$ which are mutually suitable Latin squares with constant diagonal entries. 
\begin{align*}
L_1&=\left(
\begin{array}{ccccc}
 0 & 1 & 2 & 3 & 4 \\
 4 & 0 & 1 & 2 & 3 \\
 3 & 4 & 0 & 1 & 2 \\
 2 & 3 & 4 & 0 & 1 \\
 1 & 2 & 3 & 4 & 0 \\
\end{array}
\right),\quad 
L_2=
\left(
\begin{array}{ccccc}
 0 & 2 & 4 & 1 & 3 \\
 3 & 0 & 2 & 4 & 1 \\
 1 & 3 & 0 & 2 & 4 \\
 4 & 1 & 3 & 0 & 2 \\
 2 & 4 & 1 & 3 & 0 \\
\end{array}
\right),\\
L_3&=\left(
\begin{array}{ccccc}
 0 & 3 & 1 & 4 & 2 \\
 2 & 0 & 3 & 1 & 4 \\
 4 & 2 & 0 & 3 & 1 \\
 1 & 4 & 2 & 0 & 3 \\
 3 & 1 & 4 & 2 & 0 \\
\end{array}
\right),\quad 
L_4=\left(
\begin{array}{ccccc}
 0 & 3 & 1 & 4 & 2 \\
 2 & 0 & 3 & 1 & 4 \\
 4 & 2 & 0 & 3 & 1 \\
 1 & 4 & 2 & 0 & 3 \\
 3 & 1 & 4 & 2 & 0 \\
\end{array}
\right).
\end{align*}
Fix a bijection $\varphi: \mathbb{F}_5\rightarrow \mathbb{F}_4\cup\{x\}$ such that $\varphi(0)=x$. 
We now define the incidence matrices of symmetric group divisible designs $N_{\beta}$ ($\beta\in\mathbb{F}_5^*$) by replacing $\gamma\in\mathbb{F}_4$ in $L_{\beta}$ with $C_{\varphi(\gamma)}$.
\begin{align*}
N_1&=\left(
\begin{array}{ccccc}
 C_x & C_{\varphi(1)} & C_{\varphi(2)} & C_{\varphi(3)} & C_{\varphi(4)} \\
 C_{\varphi(4)} & C_x & C_{\varphi(1)} & C_{\varphi(2)} & C_{\varphi(3)} \\
 C_{\varphi(3)} & C_{\varphi(4)} & C_x & C_{\varphi(1)} & C_{\varphi(2)} \\
 C_{\varphi(2)} & C_{\varphi(3)} & C_{\varphi(4)} & C_x & C_{\varphi(1)} \\
 C_{\varphi(1)} & C_{\varphi(2)} & C_{\varphi(3)} & C_{\varphi(4)} & C_x \\
\end{array}
\right).
\end{align*}

Then the matrices $I_{80}, I_{20}\otimes(J_{5}-I_{5}),I_5\otimes(J_4-I_4)\otimes J_4,(J_5-I_5)\otimes I_{16}, (J_5-I_5)\otimes I_{4}\otimes (J_4-I_4)$, $N_i-(J_5-I_5)\otimes I_{16}$ ($i \in \mathbb{F}_5^*$)
 form a commutative association scheme with $8$ classes. 
The first eigenmatrix $P$ is 

\begin{align*}
P=\bordermatrix{
      & R_0 & R_1 & R_2  & R_3 & R_4  & R_{5,1} & R_{5,2} & R_{5,3} & R_{5,4}  \cr
V_0 & 1 &  3 & 12 & 4 & 12 & 12 & 12 & 12 & 12 \cr
V_1 & 1 &  3 & -4 & 4 & 12 & -4 & -4 & -4 & -4 \cr
V_2 & 1 &  -1 & 0 & 4 & -3 & 0 & 0 & 0 & 0 \cr
V_3 & 1 &  3 & 12  & -1 & -4 & -3 & -3 & -3 & -3 \cr
V_4 & 1 &  -1 & 0 & -1 & 1 & 1 & 1 & 1 & 1 \cr
V_{5,1} & 1 &  -1 & 0  & -1 & 1 & 3w+1 & 3w^2+1& 3w^3+1& 3w^4+1 \cr
V_{5,2} & 1 &  -1 & 0  & -1 & 1 & 3w^3+1 & 3w+1& 3w^4+1& 3w^2+1 \cr
V_{5,3} & 1 &  -1 & 0  & -1 & 1 & 3w^2+1 & 3w^4+1& 3w+1& 3w^3+1 \cr
V_{5,4} & 1 &  -1 & 0  & -1 & 1 & 3w^4+1 & 3w^3+1& 3w^2+1& 3w+1 
},
\end{align*}
where $w=\frac{\sqrt{5}-1}{4}+\sqrt{\frac{-5-\sqrt{5}}{8}}$. 
\end{example}

\noindent {\bf Acknowledgments.}
The authors thank Sara Sasani for some computational help and the referees for their comments. 
Hadi Kharaghani is supported by an NSERC Discovery Grant.  Sho Suda is supported by JSPS KAKENHI Grant Number 15K21075, 18K03395.



\begin{thebibliography}{99}
\bibitem{B71}
L. D. Baumert,
Cyclic difference sets. 
Lecture Notes in Mathematics, Vol. 182 Springer-Verlag, Berlin-New York 1971 vi+166 pp. 

\bibitem{BJL}
T. Beth, D. Jungnickel and H. Lenz, 
Design theory. Vol. I. (English summary) 
Second edition. Encyclopedia of Mathematics and its Applications, 69. Cambridge University Press, Cambridge, 1999. xx+1100 pp.

\bibitem{B}
R. C. Bose, 
Symmetric group divisible designs with the dual property,
{\sl J.\ Stat.\ Plann. Inference} {\bf 1} (1977), 87--101.

\bibitem{HKO}
W. H. Holzmann, H. Kharaghani, W. Orrick, 
On the real unbiased Hadamard matrices.
Combinatorics and graphs, 243--250, Contemp. Math., 531, Amer. Math. Soc.,
Providence, RI, 2010.


\bibitem{KSS}
H. Kharaghani, S. Sasani and S. Suda, 
A strongly regular decomposition of the complete graph and its association scheme, 
{\sl Finite Fields Appl.} {\bf 48} (2017), 356--370.  


\bibitem{KS}
H. Kharaghani and S. Suda, 
Linked systems of symmetric group divisible designs, {\sl J. Algebraic Combin.}  {\bf 47} (2017), no. 2, 319--343.

\bibitem{KS2}
H. Kharaghani and S. Suda, Linked system of symmetric group divisible designs of type II, {\sl Des.\ Codes  Cryptogr.} {\bf 87} (2019), no. 10, 2341--2360. 

\bibitem{P}
P. Mih\u{a}ilescu, Primary cyclotomic units and a proof of Catalan's conjecture, 
{\sl J. Reine Angew. Math.} {\bf 572}, (2004) 167--195.

\bibitem{SS}
R. G. Stanton and D. A. Sprott, 
A family of difference sets,  
{\sl Canad. J. Math.} {\bf 10} (1958), 73--77. 

\bibitem{S}
D. R. Stinson, 
Combinatorial Designs: Constructions and Analysis, New York, Springer, 2004. 

\bibitem{W}
W. D. Wallis, 
Construction of strongly regular graphs using affine designs. 
{\sl Bull.\ Austral.\ Math.\ Soc.} {\bf 4} (1971), 41--49. 
\end{thebibliography}
\end{document}